\documentclass[journal,comsoc,onecolumn]{IEEEtran}

\usepackage[T1]{fontenc}

\usepackage{amssymb}

\usepackage{enumitem} 

\usepackage[usenames,dvipsnames,svgnames,table]{xcolor}
\definecolor{bananayellow}{rgb}{1.0, 0.88, 0.21}
\usepackage[prependcaption,colorinlistoftodos]{todonotes}

\usepackage{balance}

\ifCLASSINFOpdf
\else
\fi
\usepackage{graphicx}

\usepackage{amsmath}
\usepackage{color}

\usepackage[hidelinks]{hyperref}
\usepackage{xcolor}
\hypersetup{
	colorlinks,
	linkcolor={red!80!black},
	citecolor={blue!80!black},
	urlcolor={blue!80!black}
}

\usepackage{url}

\usepackage{amsthm}

\usepackage{mathrsfs}
\theoremstyle{plain}
\newtheorem{thm}{Theorem}
\newtheorem{lem}[thm]{Lemma}

\newtheorem{prob}{Problem}

\theoremstyle{definition} 
\newtheorem{defn}[thm]{Definition}

\newtheorem{rem}[thm]{Remark}

\begin{document}
%
\title{A Suboptimality Approach to\\ Distributed Linear Quadratic Optimal Control}
%
%
%
\author{Junjie~Jiao,~Harry~L.~Trentelman,~\IEEEmembership{Fellow,~IEEE,} and M.~Kanat~Camlibel,~\IEEEmembership{Member,~IEEE}

\thanks{Manuscript received \today.}

\thanks{The authors are with the Johann Bernoulli Institute for Mathematics and Computer Science, University of Groningen, Groningen, 9700 AV, The Netherlands (Email: \href{mailto:j.jiao@rug.nl}{j.jiao@rug.nl};
\href{mailto:h.l.trentelman@rug.nl}{h.l.trentelman@rug.nl}; \href{mailto:m.k.camlibel@rug.nl}{m.k.camlibel@rug.nl}).}
}
%
%
%

\markboth{Journal of \LaTeX\ Class Files,~Vol.~14, No.~8, \today} 
{Shell \MakeLowercase{\textit{et al.}}: Bare Demo of IEEEtran.cls for IEEE Communications Society Journals}
%



\maketitle

\begin{abstract}
	This paper is concerned with the distributed linear quadratic optimal control problem. 
	In particular, we consider a suboptimal version of the distributed optimal control problem for undirected multi-agent networks. 
	Given a multi-agent system with identical agent dynamics and an associated global quadratic cost functional, our objective is to design suboptimal distributed control laws that guarantee the controlled network to reach consensus and the associated cost to be smaller than an a priori given upper bound.
		We first analyze the suboptimality for a given linear system
	and then apply the results to linear multi-agent systems.
	Two design methods are then provided to compute such suboptimal distributed controllers, involving the solution of a single Riccati inequality of dimension equal to the dimension of the agent dynamics, and the smallest nonzero and the largest eigenvalue of the graph Laplacian. 
	Furthermore, we relax the requirement of exact knowledge of the smallest nonzero and largest eigenvalue of the graph Laplacian by using only lower and upper bounds on these eigenvalues.
	Finally, a simulation example is provided to illustrate our design method.
\end{abstract}

\begin{IEEEkeywords}
Distributed control, linear quadratic optimal control, suboptimality, consensus, multi-agent systems.
\end{IEEEkeywords}

%
\IEEEpeerreviewmaketitle

\section{Introduction}
\IEEEPARstart{I}{n} this paper, we study the distributed linear quadratic optimal control problem for multi-agent networks. 
In this problem, we are given a number of identical agents represented by a finite dimensional linear input-state system, and an undirected graph representing the communication between these agents. 
Given is also a quadratic cost functional that penalizes the differences between the states of neighboring agents and the size of the local control inputs. 
The distributed linear quadratic problem is then to find a distributed diffusive control law that, for given initial states of the agents, minimizes the cost functional, while achieving consensus for the controlled network. 
This problem is non-convex and difficult to solve, and it is unclear whether a solution exists in general \cite{lunze_conf_2014}.
Therefore, in this paper, instead of addressing the version formulated above, we will study a {\em suboptimal} version of the distributed optimal control problem. 
Our aim is to design suboptimal distributed diffusive control laws that guarantee the controlled network to reach consensus and the associated cost to be smaller than an a priori  given upper bound.

%
In the past, there has been work on the distributed optimal control problem before.
In \cite{jan_mm2015}, \cite{shen_zeng2017} and \cite{7862732}, it is shown that diffusive couplings are necessary for minimizing a cost functional that integrates a quadratic form involving state differences and inputs. 
However, these papers do not provide a design method for finding an optimal distributed control law.

On the other hand, there has been some work on the design of distributed diffusive control laws.
It is shown in \cite{tuna2008lqr} and \cite{hongwei_zhang2011} that, using the distributed control law derived from the solution of a local algebraic  Riccati equation, synchronization is achieved with sufficiently large coupling gain. However, no cost functionals were taken into consideration.
In \cite{tamas2008}, a design method was introduced for computing distributed suboptimal controllers, which requires the solution of a single LQR problem whose size depends on the maximum node degree of the communication graph. 
In \cite{wei_ren2010}, the authors consider a distributed optimal control problem for multi-agent systems with single integrator agent dynamics, and obtain an expression for the optimal distributed diffusive control law.
	In addition, a distributed optimal control problem was considered from the perspective of cooperative game theory  in \cite{Semsar2009}. The problem there is then solved by transforming it into a maximization problem for LMI's, taking into consideration the structure of the graph Laplacian.
For related work we also mention  \cite{lunze_ecc_2013}, \cite{guaranteed_cost}, \cite{Lunze_ijc_2014} and \cite{Fazelnia2017}, to name a few.
Also,  in \cite{Nguyen2015}, a hierarchical control approach was introduced for linear leader-follower multi-agent systems. 
For the case that the weighting matrices in the cost functional are chosen to be of a special form, two suboptimal controller design methods are given. 
In addition, in \cite{kristian2014},  an inverse optimal control problem was addressed both for leader-follower and leaderless multi-agent systems.  
For a  class of digraphs, the authors show that distributed optimal controllers exist and can be obtained if the weighting matrices are assumed to be of a special form, capturing the graph information. 
For other papers related to distributed inverse optimal control, see also \cite{huaguang_zhang2015}, \cite{nguyen_2017}.

As announced before, in this paper our objective is to design distributed diffusive control laws that guarantee the controlled network to reach consensus and the associated cost to be smaller than an  a priori given upper bound.
The main contributions of the paper are the following: 

\begin{enumerate}
\item 
	We present two design methods for computing suboptimal distributed diffusive control laws, both based on computing a positive semi-definite solution of a single Riccati inequality of dimension equal to the dimension of the agent dynamics.  
	In the computation of the local control gain, the smallest nonzero eigenvalue and the largest eigenvalue of the graph Laplacian are involved.
\item  
	For the case that exact information on the smallest nonzero eigenvalue and the largest eigenvalue of the graph Laplacian is not available, we establish a design method using only lower and upper bounds on these Laplacian eigenvalues.
\end{enumerate}

The remainder of this paper is organized as follows. 
In Section \ref{sec_notation}, we introduce the basic notation and formulate the suboptimal distributed  linear quadratic control problem. 
Section \ref{sec_single_sys} presents the analysis and design of suboptimal linear quadratic control for linear systems, 
collecting preliminary results for treating the actual suboptimal  distributed control problem for multi-agent systems. 
Then, in Section \ref{sec_mas}, we study the suboptimal distributed control problem for linear multi-agent systems.
In addition, a simulation example is provided in Section \ref{sec_simulation} to illustrate our results. 
Finally, Section \ref{sec_conclusion} concludes this paper.


\section{Notation and Problem Formulation}\label{sec_notation}

\subsection{Notation}
We denote by $\mathbb{R}$ the field of real numbers, by $\mathbb{R}^{n\times m}$ the set of $n\times m$ real matrices. 
For a given matrix $A$, its transpose and inverse (if it exists) are denoted by $A^{\top}$ and $A^{-1}$, respectively.
The identity matrix of dimension $n \times n$ is denoted by $I_n$. 
We denote the Kronecker product of two matrices $A$ and $B$ by $A\otimes B$, which has the property that $(A_1\otimes B_1)(A_2\otimes B_2) = A_1 A_2\otimes B_1 B_2 $. 
For a given symmetric matrix $P$ we denote $P>0$ if it is positive definite and $P\geq 0$ if it is positive semi-definite.
By  $\text{diag} (  a_1, a_2, \ldots, a_n )$,
we denote the $n\times n$ diagonal matrix with $a_1, a_2, \ldots, a_n $ on the diagonal.
The column vector $\mathbf{1}_n\in \mathbb{R}^n$ denotes the vector whose components are all $1$.

Throughout this paper, an undirected graph is denoted by $\mathscr{G} = (\mathscr{V},\mathscr{E})$ with nonempty finite set of $N$ nodes $\mathscr{V} = \{ v_1, v_2, \ldots, v_N \}$ and edge set $\mathscr{E} = \{ e_1, e_2, \ldots,e_M \}$.
A pair $(v_i, v_j) \in \mathscr{E}$, with $v_i, v_j\in \mathscr{V}$ and $i \neq j$, represents an edge from node $i$ to node $j$.
The graph is called undirected if $(v_i, v_j) \in \mathscr{E}$ implies $(v_j, v_i) \in \mathscr{E}$.
The neighbor set of node $i$ is denoted by $\mathscr{N}_i =\{ v_j \in \mathscr{V}:(v_i, v_j)\in \mathscr{E} \}$.
The Laplacian matrix $L$ of an undirected graph is symmetric and consequently has real eigenvalues. 
For an undirected graph, all eigenvalues of Laplacian are nonnegative and it always has $0$ as an eigenvalue. 
The graph is connected if and only if $0$ is a simple eigenvalue of $L$. 
In the sequel,  assume that $\mathscr{G}$ is connected. In that case the eigenvalues of $L$ can be ordered in increasing order as $0=\lambda_1 < \lambda_2 \leq \cdots \leq \lambda_N$ and there exists an orthogonal matrix $U$ such that 
$U^{\top}LU = \text{diag} ( 0, \lambda_2, \ldots, \lambda_N )$.
%
Moreover, there holds that $U = \left( \frac{1}{\sqrt{N}}\mathbf{1}_N\quad U_2 \right)$ and $U_2 U^{\top}_2 = I_N - \frac{1}{N}\mathbf{1}_N\mathbf{1}_N^{\top}$.

\subsection{Problem Formulation}
In this paper, we consider a multi-agent system consisting of $N$ identical agents. 
The underlying graph is assumed to be undirected and connected, and the corresponding Laplacian matrix is denoted by $L$. 
The dynamics of the identical agents is represented by
the continuous-time linear time-invariant
 (LTI) system given by 
\begin{equation}\label{sys_i}
\dot{x}_i(t) = Ax_i(t) + Bu_i(t),\quad x_i(0) = x_{i0}, \quad i=1, 2, \ldots,N
\end{equation}
where $A\in \mathbb{R}^{n\times n}$, $B\in \mathbb{R}^{n\times m}$, and $x_i\in \mathbb{R}^{n}, u_i \in \mathbb{R}^{m}$ are the state and input of the $i$-th agent, respectively. 
Throughout this paper, we assume  that the pair $(A,B)$ is stabilizable.

We consider the infinite horizon distributed linear quadratic optimal control problem for  multi-agent system (\ref{sys_i}), where the global cost functional integrates the weighted quadratic difference of states between every agent and its neighbors, and also penalizes the inputs in a quadratic form. 
Thus, the cost functional  considered in this paper is given by
\begin{equation}\label{cost_i}
J(u) = \int_{0}^{\infty}\frac{1}{2}\sum_{i=1}^{N}\sum_{j\in \mathscr{N}_i}(x_i-x_j)^{\top} Q (x_i-x_j) + \sum_{i=1}^{N}u_i^{\top}R u_i \ dt
\end{equation}
where $Q\geq 0$ and $R > 0$ are given real  weighting matrices.

We can rewrite multi-agent system (\ref{sys_i}) in compact form as
\begin{equation}\label{net_sys_1}
\dot{x} = (I_N\otimes A) x + (I_N\otimes B)u,\quad x(0) =x_0
\end{equation}
with $x = \left( x_1^{\top},\ldots,x_N^{\top} \right)^{\top}$, $u = \left( u_1^{\top},\ldots,u_N^{\top} \right)^{\top}$, where $x\in \mathbb{R}^{nN}$, $u\in \mathbb{R}^{mN}$ contain the  states and inputs of all agents, respectively.
Note that, although the agents have identical dynamics, we allow the initial states of the individual agents to differ. 
These initial states are collected in the joint vector of initial states $x_0$.
Moreover, we can also write the cost functional (\ref{cost_i}) in compact form as
\begin{equation}\label{cost_all}
J(u) = \int_{0}^{\infty} x^{\top}(L\otimes Q)x +u^{\top}(I_N\otimes R)u \ dt.
\end{equation}

The  distributed linear quadratic problem is the problem of minimizing the cost functional (\ref{cost_all}) over all distributed control laws that achieve consensus. 
By a distributed control law we mean a control law of the form 
\begin{equation}\label{controller}
u = (L\otimes K) x,
\end{equation} 
where 
$K \in \mathbb{R}^{m\times n}$ is an identical feedback gain for all agents.

By interconnecting the agents using this control law, we obtain the overall network dynamics
\begin{equation}\label{sys_closed}
\dot{x} = (I_N\otimes A + L \otimes BK) x.
\end{equation}
Foremost, we want the control law to achieve consensus:
\begin{defn}\label{def_consensus}
	We say the network reaches consensus using control law (\ref{controller}) if for all $i,j =1,2,\ldots,N$ and for all initial conditions on $x_i$ and $x_j$, we have $$x_i(t)-x_j(t) \rightarrow 0
	\text{  as  } t \rightarrow \infty.$$
\end{defn}

As a function of the to be designed feedback gain $K$, the cost functional (\ref{cost_all}) can be rewritten as
\begin{equation}\label{cost_all_new}
J(K) = \int_{0}^{\infty}x^{\top}\left(L\otimes Q + L^2\otimes K^{\top} R K\right)x  \ dt.
\end{equation}
In other words, the distributed linear quadratic control problem is to minimize (\ref{cost_all_new}) over all $K \in \mathbb{R}^{m\times n}$  
 such that the controlled network (\ref{sys_closed}) reaches consensus.

Due to the distributed nature of the control law (\ref{controller}) as imposed by the network topology, the distributed linear quadratic problem is a non-convex optimization problem. 
It is therefore difficult, if not impossible, to find a closed form solution for an optimal controller, or such optimal controller may not even exist.
Therefore, as already announced in the introduction, in this paper we will study and resolve a version of this problem involving the design of suboptimal distributed control laws. 
Specifically, we want to design distributed suboptimal controllers  of the form  (\ref{controller}) for system (\ref{net_sys_1}) such that consensus is achieved and the associated cost functional (\ref{cost_all_new}) is smaller than an a priori given upper bound. 
More concretely, we will consider the following problem:
\begin{prob}\label{prob1}
Consider  multi-agent system (\ref{net_sys_1}) with identical linear agent dynamics and given initial state $x(0) = x_0$. 
Assume the network graph is a connected undirected graph with Laplacian $L$. 
Consider the associated cost functional given by (\ref{cost_all}). 
Let $\gamma >0$ be an a priori given upper bound for the cost to be achieved.
The problem is to find a distributed controller of the form (\ref{controller}) so that the controlled network (\ref{sys_closed}) reaches consensus and the cost (\ref{cost_all_new}) associated with this controller is smaller than the given upper bound, i.e., $J(K) < \gamma$.
\end{prob}

Before we address  Problem \ref{prob1}, we first briefly discuss the suboptimal linear quadratic optimal problem for a single linear system. 
This will be the subject of the next section.

\section{Suboptimal Control for Linear Systems}\label{sec_single_sys}
In this section, we consider the linear quadratic suboptimal problem for a single linear system.  
We will first analyze the quadratic performance of a given autonomous system.
Subsequently, we will discuss how to design suboptimal control laws for a linear system with inputs.

\subsection{Suboptimality analysis for autonomous systems}
Consider the autonomous system 
\begin{equation}\label{sys_auto}
\dot{x}(t) =\bar{A}x(t),\quad x(0) = x_0
\end{equation}
where  $\bar{A}\in \mathbb{R}^{n\times n}$ and $x\in \mathbb{R}^{n}$ is the state.
We consider the quadratic performance of system (\ref{sys_auto}),  given by
\begin{equation}\label{cost_auto}
J = \int_{0}^{\infty} x^{\top}\bar{Q} x \ dt
\end{equation}
where  $\bar{Q}\geq 0$ is a given real weighting matrix.
Note that the performance $J$ is finite if  system (\ref{sys_auto}) is stable, i.e., $\bar{A}$ is Hurwitz.

We are interested in finding conditions such that the performance (\ref{cost_auto}) of system (\ref{sys_auto}) is smaller than a given upper bound. 
For this, we have the following lemma (see also \cite{algebraic_1997}, \cite{harry_book}):
\begin{lem}\label{lem_autonomous}
	Consider system  (\ref{sys_auto}) with the corresponding quadratic performance (\ref{cost_auto}). 
	The performance is finite if system (\ref{sys_auto}) is stable, i.e., $\bar{A}$ is Hurwitz.
	In this case, it is given by 
	\begin{equation}\label{quad_perf}
		J = x_0^{\top} Yx_0,
	\end{equation}
	where $Y$ is the unique positive semi-definite solution of
	\begin{equation}\label{lyap_eq}
		 \bar{A}^{\top} Y + Y\bar{A} + \bar{Q} = 0.
	\end{equation}
	Alternatively,
	\begin{equation}\label{lyap_ineq}
	\begin{split}
		J =\inf  \{ x_0^{\top} Px_0 \ | \ P \geq 0 \text{ and } \bar{A}^{\top} P + P\bar{A}  + \bar{Q} < 0 \}.
	\end{split}
	\end{equation}
	
\end{lem}

\begin{proof}
The fact that the quadratic performance (\ref{cost_auto}) is given by the quadratic expression (\ref{quad_perf})  involving the Lyapunov equation (\ref{lyap_eq}) is well-known.

We will now prove (\ref{lyap_ineq}). 
Let $Y$ be the solution to Lyapunov equation (\ref{lyap_eq}) and let $P$ be a positive semi-definite solution to the Lyapunov inequality in (\ref{lyap_ineq}). Define $ X := P - Y$.  Then we have
\begin{equation*}
	\bar{A}^{\top} (X + Y) + (X + Y) \bar{A}  + \bar{Q} < 0.
\end{equation*}
So consequently,
\begin{equation*}
	\bar{A}^{\top} X + X \bar{A}   <0.
\end{equation*}
Since $\bar{A}$ is Hurwitz, it follows that $X>0$.
Thus, we have $P > Y$ and hence $J \leq x_0^{\top}Px_0$ for any positive semi-definite solution $P$ to the Lyapunov inequality.

Next we will show that for any $\epsilon >0$ there exists a positive semi-definite matrix $P_{\epsilon}$ satisfying the Lyapunov inequality such that $P_{\epsilon} < Y+\epsilon I$, and consequently $x_0^{\top}P_{\epsilon}x_0 \leq J +\epsilon \|x_0 \|^2$. 
Indeed, for given $\epsilon$, take $P_{\epsilon}$ equal to the unique positive semi-definite solution of
\begin{equation}\label{new_lya}
\bar{A}^{\top} P + P\bar{A}  + \bar{Q} +\epsilon I= 0.
\end{equation}
Clearly then, $P_{\epsilon} =  \int_{0}^{\infty} e^{\bar{A}^{\top}t} (\bar{Q} +\epsilon I)e^{\bar{A}t}\ dt$, 
so $P_{\epsilon} \downarrow Y$ as $\epsilon \downarrow 0$. This proves our claim.
\end{proof}

The following theorem now yields {\em necessary} and {\em sufficient} conditions such that, for a given upper bound $\gamma >0$, the quadratic performance (\ref{cost_auto}) satisfies $J < \gamma$.

\begin{thm}\label{thm_autonomous}
	Consider system (\ref{sys_auto}) with the associated quadratic performance (\ref{cost_auto}).
	For given $\gamma > 0$, we have that $\bar{A}$ is Hurwitz and $J < \gamma$ if and only if there exists a positive semi-definite solution $P$ satisfying
		\begin{align}
		\bar{A}^{\top} P  + 	P\bar{A} + \bar{Q} &< 0,\label{lya_ineq_3} \\
		x_0^{\top} Px _0  &< \gamma. \label{conditions_p}
	\end{align}
\end{thm}

\begin{proof}
(if)
Since there exists a positive semi-definite solution to the Lyapunov inequality (\ref{lya_ineq_3}), it follows that $\bar{A}$ is Hurwitz.
Take a positive semi-definite matrix  $P$ satisfying
the inequalities 
(\ref{lya_ineq_3}) and (\ref{conditions_p}).
By Lemma \ref{lem_autonomous}, we then immediately have $J \leq x_0^{\top}Px_0 <\gamma$.

(only if)
If $\bar{A}$ is Hurwitz and $J <\gamma$, then, again by Lemma \ref{lem_autonomous}, there exists a positive semi-definite solution $P$ to the Lyapunov inequality (\ref{lya_ineq_3}) such that $J \leq x^{\top}_0 P x_0 < \gamma$. 
\end{proof}

\begin{rem}\label{rem_autonomous}
	Theorem \ref{thm_autonomous} provides a necessary and sufficient condition for the performance of (\ref{sys_auto}) to be less than a given upper bound.
	Given initial condition $x_0$, we can either solve equation (\ref{lyap_eq}) and compute $J$ to check whether $J < \gamma$.
	Alternatively, there exists a positive semi-definite solution to the linear matrix inequalities (\ref{lya_ineq_3}) and (\ref{conditions_p}) if and only if $J <\gamma$. 
\end{rem}

In the next subsection, we will discuss the suboptimal control problem for a linear system with inputs.

\subsection{Suboptimal control design for  linear systems with inputs}
In this section, we consider the  finite dimensional  LTI system given by
\begin{equation}\label{sys_input}
\dot{x}(t) = Ax(t) + Bu(t),\quad x(0) =x_0
\end{equation}
where $A\in \mathbb{R}^{n\times n}, B\in \mathbb{R}^{n\times m}$, and $x\in \mathbb{R}^{n}$, $u\in \mathbb{R}^{m}$ are state and input, respectively. 
Assume that the pair $(A,B)$ is stabilizable. 
The associated cost functional is given by  
\begin{equation}\label{cost_input}
J(u) = \int_{0}^{\infty} x^{\top}Q x + u^{\top}R u \ dt
\end{equation}
where $Q \geq 0$ and $R > 0$ are given weighting matrices that penalize the state and input, respectively.

We want to find a state feedback control law $u=Kx$ such that the closed system 
\begin{equation}\label{close_sys}
\dot{x}(t) = (A + BK)x(t),\quad x(0) =x_0
\end{equation}
is stable and, for a given upper bound $\gamma >0$, the corresponding cost 
\begin{equation}\label{cost_k}
J(K) = \int_{0}^{\infty} x^{\top}(Q + K^{\top} R K) x \ dt
\end{equation}
satisfies $J(K) < \gamma$.


The following theorem gives us a sufficient condition for the existence of such control law.

\begin{thm}\label{thm_single_sys}
	Consider system (\ref{sys_input}) with the associated cost functional  (\ref{cost_input}).
	Assume that the pair $(A,B)$ is stabilizable. Let $\gamma >0$.
	Suppose that there exists 
	a positive semi-definite $P$ satisfying
		\begin{align}
		A^{\top} P + PA - PBR^{-1}B^{\top}P +  Q &< 0,\label{ineq_p} \\
		x_0^{\top} P x_0 &< \gamma. \label{ini_gamma}
	\end{align}
	Let $K:=- R^{-1}B^{\top}P$. Then the controlled system (\ref{close_sys}) is  stable and the control law $u = Kx$ is suboptimal, i.e., $J(K) <\gamma$. 
\end{thm}

\begin{proof}
Substituting $K:= -R^{-1}B^{\top}P$ into (\ref{close_sys}) yields
\begin{equation}\label{closed_sys_pp}
	\dot{x}(t) = (A-BR^{-1}B^{\top}P)x(t) ,\quad x(0) =x_0.
\end{equation}
Since $P$ satisfies (\ref{ineq_p}), it should also satisfy
\begin{equation*}
(A-BR^{-1}B^{\top}P)^{\top}P + P(A-BR^{-1}B^{\top}P) + Q + PBR^{-1}B^{\top}P < 0,
\end{equation*}
which implies that $A-BR^{-1}B^{\top}P$ is Hurwitz,
i.e., the closed system (\ref{closed_sys_pp}) is stable. 
Consequently, the corresponding cost is finite and equal to
\begin{equation*}
	J(K) = \int_{0}^{\infty} x^{\top} (Q + PBR^{-1}B^{\top}P) x  \ dt.
\end{equation*}
Since (\ref{ini_gamma}) holds, by taking $\bar{A} = A-BR^{-1}B^{\top}P$ and $\bar{Q} = Q + PBR^{-1}B^{\top}P$ in Theorem \ref{thm_autonomous}, we immediately have $J(K) < \gamma$.
\end{proof}


\begin{rem}
	Theorem \ref{thm_single_sys} provides a method to find a class of suboptimal control laws satisfying $J(K)<\gamma$. 
	Choosing the feedback gain as $K := -R^{-1}B^{\top}P$  is one possible choice for  such suboptimal control laws.
	For other design methods see also \cite{guaranteed_cost}.
\end{rem}

\begin{rem}
	Note that the suboptimal control design given in Theorem \ref{thm_single_sys} is more flexible than the optimal control design. Any positive semi-definite matrix $P$  satisfying inequalities (\ref{ineq_p}) and (\ref{ini_gamma}) makes control law $u=Kx$ with $K = -R^{-1}B^{\top}P$ suboptimal with respect to $J(K) <\gamma$. 
\end{rem}


In the next section, we will show how to apply the above design method for suboptimal control to the distributed linear quadratic control problem for multi-agent systems.

\section{Suboptimal Control Design for Linear Multi-Agent Systems}\label{sec_mas}
In this section, we consider the distributed linear quadratic control problem for a multi-agent system consisting of $N$ agents with identical finite dimensional LTI system.

As in Section II, the dynamics of the identical agents is represented by
\begin{equation}
\dot{x}_i(t) = Ax_i(t) + Bu_i(t), \quad x_{i}(0) =x_{i0}, \quad i = 1, 2, \ldots,N
\end{equation}
where $A\in \mathbb{R}^{n\times n}$, $B\in \mathbb{R}^{n\times m}$, and $x_i\in \mathbb{R}^{n}, u_i \in \mathbb{R}^{m}$ are the state and input of $i$-th agent, respectively.
Assume that the pair $(A,B)$ is stabilizable, and the underlying graph is an undirected connected graph with  corresponding Laplacian denoted by $L$.

Denoting $x = \left( x_1^{\top},\ldots,x_N^{\top} \right)^{\top}$, $u = \left( u_1^{\top},\ldots,u_N^{\top} \right)^{\top}$, 
we can rewrite the  multi-agent system in compact form as
\begin{equation}\label{net_sys}
\dot{x} = (I_N\otimes A) x + (I_N\otimes B)u,\quad x(0) =x_0.
\end{equation}

The cost functional we consider was already introduced in (\ref{cost_all}). We repeat it here for convenience:
\begin{equation}\label{cost_all_2}
J(u) = \int_{0}^{\infty}x^{\top}(L\otimes Q)x +u^{\top}(I_N\otimes R)u \ dt
\end{equation}
where  $Q\geq 0$ and $R > 0$ are given real  weighting matrices.

As already formulated in Problem \ref{prob1}, given a desired upper bound $\gamma >0$ for multi-agent system (\ref{net_sys}) with given initial state $x(0) = x_0$, we want to design a control law of the form 
\begin{equation}\label{control_dis}
	u=(L\otimes K)x
\end{equation} 
where $K \in \mathbb{R}^{m\times n}$ is an identical feedback gain for all agents,
such that the controlled network 
\begin{equation}\label{net_closed}
\dot{x} = (I_N\otimes A + L \otimes BK) x
\end{equation}
reaches consensus  and, moreover, the associated cost 
\begin{equation}\label{cost_kk}
J(K) = \int_{0}^{\infty}x^{\top}\left(L\otimes Q +L^2\otimes K^{\top} R K\right)x  \ dt
\end{equation}
is smaller than the given upper bound, i.e., $J(K)<\gamma$.

Let the matrix $U\in \mathbb{R}^{N\times N}$ be an orthogonal matrix that diagonalizes the Laplacian $L$. 
Define $\Lambda : = U^{\top}L U = \text{diag} ( 0, \lambda_2,\ldots,\lambda_N )$.
%
To simplify the problem given above, by applying the state and input transformations $\bar{x} =(U^{\top}\otimes I_n)x$ and $\bar{u} =(U^{\top}\otimes I_m)u$ with $\bar{x} = \left( \bar{x}_1^{\top},\ldots,\bar{x}_N^{\top} \right)^{\top}$, $\bar{u} = \left( \bar{u}_1^{\top},\ldots,\bar{u}_N^{\top} \right)^{\top}$,  system (\ref{net_sys})  becomes
\begin{equation}\label{new_newtwork_sys}
\dot{\bar{x}} = (I_N\otimes A) \bar{x} + (I_N\otimes B)\bar{u},\quad \bar{x}(0) =\bar{x}_{0},
\end{equation}
with $\bar{x}_0 = (U^{\top} \otimes I_n)x_0$.
Clearly,  (\ref{control_dis}) 
is transformed to
\begin{equation}\label{control_law}
\bar{u} = (\Lambda \otimes K)\bar{x},
\end{equation}
and the controlled network (\ref{net_closed}) transforms to
\begin{equation}\label{clsed_sys}
	\dot{\bar{x}} = \left(I_N\otimes A + \Lambda \otimes BK\right) \bar{x} .
\end{equation}
In terms of the transformed variables, the cost (\ref{cost_kk}) is given by
\begin{equation}\label{new_newtwork_cost}
{J}(K) = \int_{0}^{\infty}\sum_{i=1}^{N} \bar{x}_i^{\top} (\lambda_i Q + \lambda_i^2 K^{\top} R K) \bar{x}_i\ dt.
\end{equation}
Note that the transformed states $\bar{x}_i$ and inputs $\bar{u}_i$, $i = 2,3,\ldots, N$ appearing in system (\ref{clsed_sys}) and cost (\ref{new_newtwork_cost}) are decoupled from each other, so that we can write  system (\ref{clsed_sys}) and cost  (\ref{new_newtwork_cost}) as
\begin{align}
	\dot{\bar{x}}_1 &= A\bar{x}_1,\label{sys1} \\
	\dot{\bar{x}}_i &= (A +\lambda_i BK)\bar{x}_i ,\quad i = 2,3,\ldots,N, \label{closed_loop}
\end{align}
and
\begin{equation}\label{cost_cost}
{J}(K)= \sum_{i=2}^{N}{J}_i(K)
\end{equation}
with
\begin{equation}\label{coco}
{J}_i(K) = \int_{0}^{\infty} \bar{x}_i^{\top} (\lambda_i Q + \lambda_i^2 K^{\top}RK) \bar{x}_i\ dt, \quad i = 2, 3, \ldots,N.
\end{equation}
Note that $\lambda_1 = 0$, and that therefore (\ref{sys1}) does not contribute to the cost $J(K)$.

We first record a well-known fact (see  \cite{zhongkui_li_unified_2010}, \cite{harry_2013}) that we will use later:
\begin{lem}\label{lem_stable_consensus}
	Consider the multi-agent system with identical agent dynamics (\ref{net_sys}). Assume that the network graph is undirected and connected. Then the controlled network reaches consensus with control law (\ref{control_dis}) if and only if, for $i=2,3,\ldots,N$, systems (\ref{closed_loop}) are stable.
\end{lem}

Thus, we have transformed the problem of distributed suboptimal control for system (\ref{net_sys}) into the problem of finding a feedback gain  $K\in \mathbb{R}^{m\times n}$ such that the systems (\ref{closed_loop}) are stable and $J(K) < \gamma$. 
Moreover, since the pair $(A, B)$ is stabilizable, there exists such a feedback gain $K$ \cite{harry_2013}.

%
%

The following lemma gives  a {\em necessary} and {\em sufficient} condition for a given  feedback gain $K$ to make all systems (\ref{closed_loop}) stable and to satisfy $J(K) < \gamma$.

\begin{lem}\label{lem_n_riccati}
	Let $K$ be a feedback gain.
	Consider the systems  (\ref{closed_loop}) with  associated cost functionals (\ref{cost_cost}) and (\ref{coco}). 
	Let $\gamma>0$.
	Then all systems (\ref{closed_loop}) are stable and ${J}(K)<\gamma$ if and only if 
	there exist positive semi-definite matrices $P_i$ satisfying
		\begin{align}
		(A +  \lambda_i B K)^{\top}P_i + P_i (A + \lambda_i B K)  + \lambda_i Q +\lambda_i^2 K^{\top}R K  &<0, \label{are_ineq}\\
		\sum_{i=2}^{N} \bar{x}_{i0}^{\top} P_i \bar{x}_{i0} & <\gamma, \label{initial_condition_n_1}
	\end{align}
   for $i= 2, 3, \ldots, N$, respectively. 
\end{lem}

\begin{proof}
(if)
	Since (\ref{initial_condition_n_1}) holds, there exist sufficiently small $\epsilon_i > 0, i=2,\ldots,N$ such that $\sum_{i=2}^{N} \gamma_i < \gamma$  where $\gamma_i := \bar{x}_{i0}^{\top} P_i \bar{x}_{i0} +\epsilon_i$. 
	Because there exists $P_i$ such that (\ref{are_ineq}) and $\bar{x}_{i0}^{\top} P_i \bar{x}_{i0} < \gamma_i$ holds for all $i =2,\ldots,N$, by taking $\bar{A} = A + \lambda_i BK$ and $\bar{Q} = \lambda_i Q +\lambda_i^2 K^{\top}R K$, it follows from Theorem \ref{thm_autonomous} that all systems (\ref{closed_loop}) are stable and $J_i(K) <\gamma_i$ for $i = 2,\ldots,N$. 
	Since $J(K) = \sum_{i=2}^{N} J_i(K)$, this implies that $J(K) < \sum_{i=2}^{N} \gamma_i < \gamma$.
	
(only if)
	Since $J(K)<\gamma$ and $J(K) = \sum_{i=2}^{N} J_i(K)$, there exist sufficiently small $\epsilon_i > 0, i=2,\ldots,N$ such that $\sum_{i=2}^{N} \gamma_i < \gamma$  where $\gamma_i := J_i(K)+\epsilon_i$.
	Because all systems (\ref{closed_loop}) are stable and $J_i(K) <\gamma_i$ for $i=2,\ldots, N$, by taking $\bar{A} = A + \lambda_i BK$ and $\bar{Q} = \lambda_i Q +\lambda_i^2 K^{\top}R K$, it follows from Theorem \ref{thm_autonomous} that there exist positive semi-definite $P_i$ such that (\ref{are_ineq})  and  $\bar{x}_{i0}^{\top} P_i \bar{x}_{i0} < \gamma_i$ hold for all $i =2,\ldots,N$. 
	Since $\sum_{i=2}^{N} \gamma_i < \gamma$, this implies that $\sum_{i=2}^{N} \bar{x}_{i0}^{\top} P_i \bar{x}_{i0}  <\sum_{i=2}^{N} \gamma_i  <\gamma$.
\end{proof}

%
%
%
Lemma \ref{lem_n_riccati} establishes a {\em necessary} and {\em sufficient} condition for a given feedback gain $K$ to stabilize all systems (\ref{closed_loop}) and to satisfy $J(K)<\gamma$. 
However, Lemma \ref{lem_n_riccati}  does yet not provide a method to compute such $K$.
To this end, the following two theorems present two design methods for $K$ and, correspondingly, two suboptimal distributed control laws for multi-agent system (\ref{net_sys}) with cost functional (\ref{cost_kk}). 


\begin{thm} \label{Main1}
Consider   multi-agent system (\ref{net_sys}) with  associated cost functional (\ref{cost_kk}). Assume that the underlying graph is undirected and connected. Let $\gamma>0$.  
Choose $c$ such that
\begin{equation}\label{c1}
\frac{2}{\lambda_2+\lambda_N} \leq c < \frac{2}{\lambda_N}.
\end{equation} 
Then there exists a positive semi-definite matrix $P$ satisfying the Riccati inequality
\begin{equation} \label{one_are1}
	A^{\top}P + PA +(c^2\lambda_N^2-2c\lambda_N)PBR^{-1}B^{\top}P +\lambda_N Q < 0.
\end{equation}
Assume, moreover, that $P$ can be found such that
\begin{equation}\label{p_N1}
	x_0^{\top}\left(\left(I_N - \frac{1}{N}\mathbf{1}_N\mathbf{1}_N^{\top}\right)\otimes P\right) x_0 < \gamma.
\end{equation}
Let $K := -cR^{-1}B^{\top}P$. Then the controlled network (\ref{net_closed}) reaches consensus and the control law (\ref{control_dis}) is suboptimal, i.e., $J(K) <\gamma$.
\end{thm}

\begin{proof}
Using the upper and lower bounds on $c$ given by (\ref{c1}), it can be verified that $c^2\lambda_i^2-2c\lambda_i \leq c^2\lambda_N^2-2c\lambda_N <0$ for  $i=2,3,\ldots,N$. 
Since also $\lambda_i \leq \lambda_N$, $P$ is a solution to the $N-1$ Riccati inequalities
\begin{equation}\label{n_are1}
	A^{\top}P + PA +(c^2\lambda_i^2-2c\lambda_i)PBR^{-1}B^{\top}P +\lambda_i Q < 0,\quad i=2,\ldots,N.
\end{equation}
Equivalently, $P$ also satisfies
the Lyapunov inequalities  
\begin{equation}\label{lya_ineq}
\begin{aligned}
	&(A-c\lambda_i BR^{-1}B^{\top}P)^{\top}P + P(A- c\lambda_iBR^{-1}B^{\top}P)  \\
	&\qquad + \lambda_i Q + c^2\lambda_i^2 PBR^{-1}B^{\top}P < 0,\quad i=2,\ldots,N.
\end{aligned}
\end{equation}

Next, recall that $\bar{x} = (U^{\top}\otimes I_n)x$ with $U = \left( \frac{1}{\sqrt{N}}\mathbf{1}_N\quad U_2 \right)$.  
From this it is easily seen that $(\bar{x}_{20}^{\top}, \bar{x}_{30}^{\top}, \cdots , \bar{x}_{N0}^{\top})^{\top} = (U_2^{\top} \otimes I_n)x_0 $. 
Also, $U_2 U^{\top}_2 = I_N - \frac{1}{N}\mathbf{1}_N\mathbf{1}_N^{\top}$.
Since  (\ref{p_N1}) holds, we have
\begin{align*}
	x_0^{\top} \left( U_2 U_2^{\top} \otimes P \right)x_0 &<\gamma \\
\Leftrightarrow \quad
	((U_2^{\top} \otimes I_n)x_0)^{\top} \left( I_{N-1} \otimes P \right) ((U_2^{\top} \otimes I_n)x_0) &<\gamma \\
\Leftrightarrow \quad
		(\bar{x}_{20}^{\top}, \bar{x}_{30}^{\top}, \cdots , \bar{x}_{N0}^{\top})
	\left(I_{N-1}\otimes P\right)
	(\bar{x}_{20}^{\top}, \bar{x}_{30}^{\top}, \cdots , \bar{x}_{N0}^{\top})^{\top} &<\gamma,
\end{align*}
which is equivalent to  
\begin{equation}\label{xbar}
	\sum_{i=2}^{N} \bar{x}_{i0}^{\top} P \bar{x}_{i0} <\gamma.
\end{equation}

Taking $P_i = P$ for $i = 2,3,\ldots,N$ and $K := -cR^{-1}B^{\top}P$ in inequalities (\ref{are_ineq}) and (\ref{initial_condition_n_1}) immediately gives us inequalities (\ref{lya_ineq}) and (\ref{xbar}).
Then it follows from Lemma \ref{lem_n_riccati} that all systems (\ref{closed_loop}) are stable and $J(K)<\gamma$. Furthermore, it follows from Lemma \ref{lem_stable_consensus} that the controlled network (\ref{net_closed}) reaches consensus and $J(K)<\gamma$.
\end{proof}

\begin{rem}\label{rem_main1}
	Theorem \ref{Main1} states that that after choosing $c$ satisfying (\ref{c1}) and positive semi-definite $P$ satisfying (\ref{one_are1}), the distributed control law with local gain $K = -c R^{-1}B^{\top}P$ is suboptimal for all initial states of the network that satisfy the inequality (\ref{p_N1}). 
	The question then arises: how should we choose $c$ and $P$ such that this local gain is suboptimal for as many initial states as possible? 
	By writing $x_0 = ( x_{10},x_{20}, \ldots,x_{N0} )$, it is easily seen that (\ref{p_N1}) is equivalent to
	\begin{equation}\label{xixj}
		\frac{1}{N}\sum_{i=1}^{N}\sum_{j>i}^{N}(x_{i0}-x_{j0})^{\top}P(x_{i0}-x_{j0}) <\gamma.
	\end{equation}
 	In other words, the smaller $P$, the bigger the differences between the local initial states are allowed to be, while still leading to suboptimality with respect to $\gamma$. 
 	In other words, we should try to find $P$ as small as possible.
  	In fact, one can find a positive definite solution $P(c,\epsilon)$ to (\ref{one_are1}) by solving the Riccati equation
	\begin{equation}\label{are}
		A^{\top}P + PA -PB\bar{R}^{-1}B^{\top}P +\bar{Q}= 0
	\end{equation} 
	with  $\bar{R}(c) = \frac{1}{-c^2\lambda_N^2+2c\lambda_N}R$ and $\bar{Q}(\epsilon) =\lambda_N Q +\epsilon I_n $	where  $c$ is chosen as in (\ref{c1}) and    $\epsilon >0$.
	If $c_1$ and $c_2$ as in (\ref{c1}) satisfy $c_1 \leq c_2$, then we have
	$\bar{R}(c_1) \leq \bar{R}(c_2)$,
	so, clearly, $P(c_1,\epsilon) \leq P(c_2,\epsilon)$.
	Similarly, if $0 <  \epsilon_1 \leq \epsilon_2$,  we immediately have  $\bar{Q}(\epsilon_1) \leq \bar{Q}(\epsilon_2)$. Again, it follows that $P(c,\epsilon_1) \leq P(c,\epsilon_2)$.
	Therefore, if we choose $\epsilon>0$ very close to $0$ and $c = \frac{2}{\lambda_2+\lambda_N}$,  we find the `best' solution to the Riccati inequality (\ref{one_are1}) in the sense explained above.
\end{rem}

Theorem \ref{Main1} provides a method to find a suboptimal distributed control law for particular choices of the parameter $c$. 
In fact, such $c$ can be also chosen in another way, which is shown in the next theorem:
\begin{thm} \label{Main2}
	Consider multi-agent system (\ref{net_sys}) with associated cost functional (\ref{cost_kk}). 
	Assume that the underlying graph  is undirected and connected. 
	Let $\gamma >0$. 
	Choose $c$ such that
	\begin{equation}\label{c2}
	0 < c <\frac{2}{\lambda_2+\lambda_N}.
	\end{equation} 
	Then there exists a positive semi-definite $P$ satisfying Riccati inequality
	\begin{equation} \label{one_are}
	A^{\top}P + PA +(c^2\lambda_2^2-2c\lambda_2)PBR^{-1}B^{\top}P +\lambda_N Q < 0.
	\end{equation}
	Assume, moreover, that $P$ can be found such that

	\begin{equation}\label{p_N}	
	x_0^{\top}\left(\left(I_N - \frac{1}{N}\mathbf{1}_N\mathbf{1}_N^{\top}\right)\otimes P\right) x_0 < \gamma. 
	\end{equation}
	Let $K := -cR^{-1}B^{\top}P$. Then the  controlled network  (\ref{net_closed}) reaches consensus and the control law (\ref{control_dis}) is suboptimal, i.e., $J(K)<\gamma$.
\end{thm}

\begin{proof}
	The proof is similar to the proof of Theorem \ref{Main1} and hence is omitted here.
\end{proof}

\begin{rem}\label{rem_main2}
	Theorem \ref{Main2} states that that after choosing $c$ satisfying (\ref{c2}) and positive semi-definite $P$ satisfying (\ref{one_are}), the distributed control law with local gain $K = -c R^{-1}B^{\top}P$ is suboptimal for all initial states of the network that satisfy the inequality (\ref{p_N}). 
	Again, the question then arises: how should we choose $c$ and $P$ such that this local gain is suboptimal for as many initial states as possible? 
	Following the idea in Remark \ref{rem_main1}, if we choose $\epsilon>0$ very close to $0$ and $c>0$ very close to $\frac{2}{\lambda_2+\lambda_N}$,  we find the `best' solution to the Riccati inequality (\ref{one_are}) in the sense as explained in Remark \ref{rem_main1}.
\end{rem}

Note that, in Theorem \ref{Main1} and Theorem \ref{Main2}, in order to compute a suitable feedback gain $K$, one needs to know $\lambda_2$ and $\lambda_N$, the smallest nonzero eigenvalue (the algebraic connectivity) and the largest eigenvalue of the graph Laplacian, exactly. 
This requires so-called global information on the network graph which might not always be available. 
There exist algorithms to estimate $\lambda_2$ in a distributed way, yielding lower and upper bounds, see e.g. \cite{ARAGUES20143253}. 
Moreover, also an upper bound for $\lambda_N$ can be obtained in terms of the maximal node degree of the graph, see \cite{eigenvaluesL}. 
Then the question arises: can we still find a suboptimal controller reaching consensus, using as information only a lower bound for $\lambda_2$ and an upper bound for $\lambda_N$?
The answer to this question is affirmative, as shown in the following theorem. 
\begin{thm}\label{Main3}
Let a lower bound for $\lambda_2$ be given by $l_2 > 0$ and an upper bound for $\lambda_N$ be given by $L_N$. 
Let $\gamma > 0$. Choose $c$ such that 
\begin{equation}\label{c3}
	\frac{2}{l_2+L_N} \leq c < \frac{2}{L_N}. 
\end{equation}
Define $K=-cR^{-1}B^{\top}P$, 
with $P\geq 0$ satisfying
\begin{equation}\label{pp}
	A^{\top}P + PA +(c^2L_N^2-2cL_N)PBR^{-1}B^{\top}P +L_N Q < 0.
\end{equation} 
Then the controlled network reaches consensus and the distributed control law $L \otimes K$ is suboptimal, i.e. $J(K) <\gamma$, for all initial states $x_0$ that satisfy \eqref{p_N1}, equivalently \eqref{xixj}.

Furthermore, if we choose $c$ such that 
\begin{equation}\label{c4}
	0 < c <\frac{2}{l_2+L_N}.
\end{equation}
and define $K=-cR^{-1}B^{\top}P$ 
with $P \geq 0$ satisfying
\begin{equation}\label{p}
	A^{\top}P + PA +(c^2 l_2^2-2c l_2)PBR^{-1}B^{\top}P +L_N Q < 0,
\end{equation} 
then, still, the controlled network reaches consensus and the distributed control law $L \otimes K$ is  suboptimal for all $x_0$ that satisfy \eqref{p_N}, equivalently \eqref{xixj}.
\end{thm}

\begin{proof}
A proof can be given along the lines of the proofs of Theorem \ref{Main1} and Theorem \ref{Main2}.
\end{proof}		

\begin{rem}
	Note that also in Theorem \ref{Main3} the question arises how to choose $c>0$ and $P\geq 0$ such that the local gain is suboptimal for as many initial states $x_0$ as possible. 
	Following the same idea as in Remark \ref{rem_main1} and Remark \ref{rem_main2}, if we choose $\epsilon >0$ very close to $0$ and $c>0$ equal to $\frac{2}{l_2 +L_N}$ in \eqref{pp} (respectively very close to $\frac{2}{l_2 +L_N}$ in \eqref{p}), we find the `best' solution to the Riccati inequalities (\ref{pp}) and (\ref{p}).
	
	Moreover, one may also ask the question: can we compare, with the same choice for $c$, solutions  to (\ref{pp}) with solutions to (\ref{one_are1}), and also solutions to (\ref{p}) with solutions to (\ref{one_are})?
	The answer is affirmative. 
	Choose $c$ that satisfies both conditions (\ref{c1}) and (\ref{c3}). One can then check that the computed positive semi-definite solution to (\ref{pp}) is indeed `larger' than that to (\ref{one_are1}) as explained in Remark \ref{rem_main1}. A similar remark holds for the positive semi-definite solutions to (\ref{p}) and  corresponding solutions to (\ref{one_are}) if $c$ satisfies both (\ref{c2}) and (\ref{c4}). 
	We conclude that if,  instead of using the exact values $\lambda_2$ and $\lambda_N$, we use a lower bound, respectively upper bound for these eigenvalues, then the computed distributed control law is suboptimal for `less' initial values of the agents.
\end{rem}	

%
%
%
\section{Illustrative Example}\label{sec_simulation}
In this section we use a simulation example borrowed from \cite{Nguyen2015} to illustrate the proposed design method for suboptimal distributed controllers. 
Consider a group of 8 linear oscillators with identical dynamics
\begin{equation}\label{oscillators}
\dot{x}_i = A x_i +B u_i, \quad x_i(0) =x_{i0}, \quad i = 1,\ldots ,8
\end{equation}
with 
\begin{equation*}
A = 
\begin{pmatrix}
	0 & 1\\
	-1 & 0
\end{pmatrix}, 
\quad
B = 
\begin{pmatrix}
	0 \\ 1
\end{pmatrix}.
\end{equation*}
Assume the underlying graph is the undirected line graph with Laplacian matrix
\begin{equation*}
L = \begin{pmatrix}
 1 & -1 & 0 & 0 & 0 & 0 & 0 & 0 \\
-1 & 2 & -1 & 0 & 0 & 0 & 0 & 0\\
0 & -1 & 2 & -1 & 0 & 0 & 0 & 0 \\
0 & 0 & -1 & 2 & -1 & 0 & 0 & 0 \\
0 & 0 & 0 & -1 & 2 & -1 & 0 & 0 \\
0 & 0 & 0 & 0 & -1 & 2 & -1 & 0 \\
0 & 0 & 0 & 0 & 0 & -1 & 2 & -1\\
 0 & 0 & 0 & 0 & 0 & 0 & -1 & 1
\end{pmatrix}.
\end{equation*}
We consider the cost functional
\begin{equation}\label{cost_exam}
	J(u) = \int_{0}^{\infty}x^{\top}(L\otimes Q)x +u^{\top}(I_{8}\otimes R)u \ dt
\end{equation} 
where the matrices $Q$ and $R$ are chosen to be
\begin{equation*}
	Q =\begin{pmatrix}
		2 & 0 \\ 0 & 1
	\end{pmatrix},\quad
	R = 1.
\end{equation*}
Let the desired upper bound for the cost functional (\ref{cost_exam}) be given as $\gamma =3$.
Our goal is to design a control law $u =(L\otimes K)x$ such that the controlled network reaches consensus and the associated cost is less than $\gamma = 3$.

In this example, we adopt the control design method given in Theorem \ref{Main1}.
The smallest nonzero and largest eigenvalue of $L$ are $\lambda_2 = 0.0979$ and $\lambda_{8} = 3.8478$.
First, we compute a positive semi-definite solution $P$ to (\ref{one_are1}) by solving the Riccati equation
\begin{equation}\label{are8}
A^{\top}P + PA +(c^2\lambda_8^2-2c\lambda_8)PBR^{-1}B^{\top}P +\lambda_{8} Q + \epsilon I_2  = 0
\end{equation}
with $\epsilon >0$ chosen small as mentioned in Remark \ref{rem_main1}. 
Here we choose $\epsilon = 0.001$.
Moreover, we choose $c = \frac{2}{\lambda_2 +\lambda_8} = 0.5$, which is the `best' choice as mentioned in Remark \ref{rem_main1}.
Then, by solving (\ref{are8}) in Matlab, we obtain 
\begin{equation*}
	P = \begin{pmatrix}
   12.1168 &    3.1303\\
3.1303 &    8.3081
	\end{pmatrix}.
\end{equation*}
Correspondingly, the local feedback gain is then equal to  $K = \begin{pmatrix}  -1.5652  & -4.1541 \end{pmatrix}$.

The corresponding distributed diffusive control law is now suboptimal (with respect to the given $\gamma$) for all initial conditions $x_0$ that satisfy the inequality
\begin{equation}
x_0^{\top}\left(\left(I_{8} - \frac{1}{{8}}\mathbf{1}_{8}\mathbf{1}_{8}^{\top}\right)\otimes P\right) x_0 < 3,
\end{equation}
which is equivalent to 
\begin{equation*}
	\frac{1}{8}\sum_{i=1}^{8}\sum_{j>i}^{8}(x_{i0}-x_{j0})^{\top}P(x_{i0}-x_{j0}) <3,
\end{equation*}
which, for example, is satisfied by the inital conditions
$x_{10}^{\top} = 
\begin{pmatrix} 
	-0.08 & 0.11
\end{pmatrix}$,
$x_{20}^{\top} = 
\begin{pmatrix} 
 0.12 & -0.08
\end{pmatrix}$,
$x_{30}^{\top} = 
\begin{pmatrix} 
-0.09 & -0.14
\end{pmatrix}$,
$x_{40}^{\top} = 
\begin{pmatrix} 
-0.12 & 0.04
\end{pmatrix}$,
$x_{50}^{\top} = 
\begin{pmatrix} 
0.07 & -0.16
\end{pmatrix}$,
$x_{60}^{\top} = 
\begin{pmatrix} 
-0.21 & 0.12
\end{pmatrix}$,
$x_{70}^{\top} = 
\begin{pmatrix} 
0.15 & -0.22
\end{pmatrix}$,
$x_{80}^{\top} = 
\begin{pmatrix} 
-0.17 & -0.14
\end{pmatrix}$.
The plots of the eight decoupled oscillators without control are shown in Figure \ref{decoupled}.
\begin{figure}[t!]
	\centering
	\includegraphics[width=8.5cm]{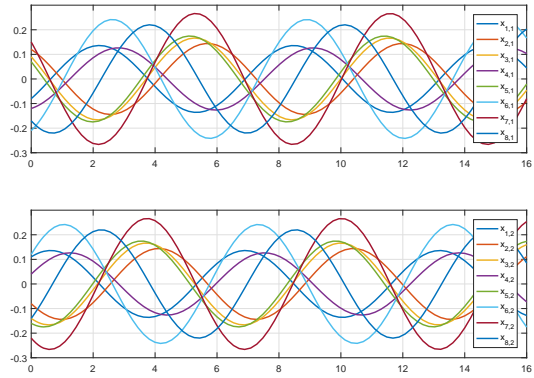}
	\caption{Plots of the state vector $x^1 = (x_{1,1},\ldots, x_{8,1})$ (upper plot) and $x^2 = (x_{1,2},\ldots, x_{8,2})$ (lower plot) of the 8 decoupled  oscillators without control} \label{decoupled}
\end{figure}
Figure \ref{consensus} shows that the controlled network of oscillators reaches consensus.
\begin{figure}[t!]
	\centering
	\includegraphics[width=8.5cm]{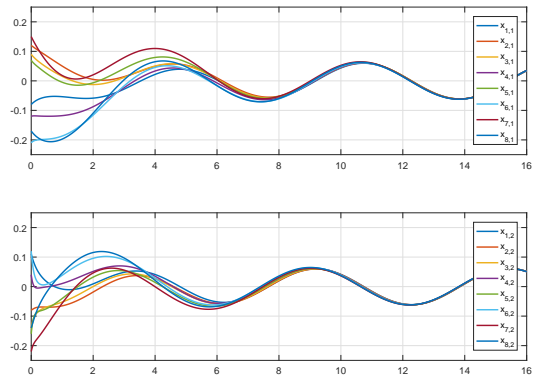}
	\caption{Plots of the state vector $x^1 = (x_{1,1},\ldots, x_{8,1})$ (upper plot) and $x^2 = (x_{1,2},\ldots, x_{8,2})$ (lower plot) of the controlled oscillator network} \label{consensus}
\end{figure}

\balance
\section{Conclusion}\label{sec_conclusion}
In this paper, we have studied a suboptimal distributed linear quadratic control problem for undirected linear multi-agent networks. 
Given a multi-agent system with identical linear agent dynamics and an associated global quadratic cost functional, we provide two design methods for computing suboptimal distributed diffusive control laws such that the controlled network is guaranteed to reach consensus and the associated cost is smaller than a given upper bound for suitable initial conditions. 
The computation of the local gain involves finding solutions of a single Riccati inequality, whose dimension is equal to the dimension of the agent dynamics, and also involves the smallest nonzero and largest eigenvalue of the graph Laplacian. 
As an extension, we remove the requirement of having exact knowledge on the smallest nonzero and largest eigenvalue of the graph Laplacian by, instead, using only lower and upper bounds for these eigenvalues.

\ifCLASSOPTIONcaptionsoff
  \newpage
\fi



%
%
%

\bibliographystyle{ieeetran}
\bibliography{suboptimal}

%

%
%
%




\end{document}